\documentclass[11pt]{amsart}
\usepackage{amsmath,amssymb,epsfig,color}%slashbox}
\newtheorem{theorem}{Theorem}[section]

\newtheorem{corollary}[theorem]{Corollary}
\newtheorem{definition}[theorem]{Definition}
\newtheorem{lemma}[theorem]{Lemma}
\newtheorem{proposition}[theorem]{Proposition}
\newtheorem{remark}[theorem]{Remark}
\newtheorem{example}[theorem]{Example}

\newcommand{\vanish}[1]{}\parskip=12pt

\def\p{\prime}

\def\R{\mathcal{R}}
\def\Z{\mathbb{Z}}
\def\B{\backslash}
\def\C{\mathcal{C}}
\def\D{\mathcal{D}}

\def\H{\mathcal{H}}
\numberwithin{equation}{section}

\begin{document}

\title[Relative Tutte Polynomials for Colored Graphs]
{Relative  Tutte Polynomials for Colored
Graphs
and Virtual Knot Theory}
\author{Y. Diao and G. Hetyei}
\address{Department of Mathematics and Statistics, UNC Charlotte,
    Charlotte, NC 28223}
\email{ydiao@uncc.edu, ghetyei@uncc.edu}
%\thanks{$^\dag$supported by NSF grant DMS-0712958.
%$^\ast$supported by NSA grant \# H98230-07-1-0073} \dedicatory{}

\begin{abstract}
We introduce the concept of a relative Tutte
polynomial of colored graphs. We show that this relative Tutte
polynomial can be computed in a way similar to the classical spanning
tree expansion used by Tutte in his original paper on this subject. We
then apply the relative Tutte polynomial to virtual knot theory. More
specifically, we show that the Kauffman bracket polynomial (hence
the Jones polynomial) of a virtual knot can be computed from the
relative Tutte polynomial of its face (Tait) graph with some
suitable variable substitutions. Our method offers an alternative
to the ribbon graph approach, using the face graph obtained from the
virtual link diagram directly.
\end{abstract}

\maketitle
\section{Introduction}

In this paper, we introduce and study the {\em relative Tutte polynomial} for a colored graph, a generalization of the (ordinary) Tutte polynomial for a colored graph. Let $G$ be a connected
graph. An important property that a classical Tutte polynomial possesses
is that it can be computed through contracting and deleting the edges of
$G$ in an arbitrary order. Assume now that $\H$ is a subset of the edges of $G$. We would like to define the relative Tutte polynomial $T_\H$ of $G$ with respect to $\H$ so that it would have such similar property, but only for the edges of $G$ not in $\H$. More specifically, we would like to be able to compute the relative Tutte polynomial $T_\H(G)$ through two stages: first through the contracting/deleting process on edges of $G$ that are not in $\H$, then assign variables to the remaining graphs (whose edges are all from $\H$) with a totally different rule. The variable assigning rule in the second stage can be quite arbitrary and that is what makes the relative Tutte polynomial more general and different from the (ordinary) Tutte polynomial.

\medskip
Tutte defined his polynomial of an un-colored graph~\cite{Tu} in
terms of counting {\em activities} with respect to a specific
labeling of the edges of the graph, and his main result in \cite{Tu} is showing that the polynomial he introduced can be computed through a spanning tree expansion by counting activities of the edges with respect to the spanning trees and a given labeling of the edges and that the polynomial is actually independent of the labeling, thus truly an invariant of the graph. This is equivalent to saying that the Tutte polynomial can be computed through the contracting/deleting process and the order of edges appearing in this process does not matter. The greatest challenge in generalizing Tutte's polynomial to colored graphs is to preserve the independence of the labeling. This challenge is
typically met by considering the Tutte polynomial of a colored
graph as an element of a polynomial ring modulo certain relations
between the variables. The most general result here is due
to Bollob\'as and Riordan~\cite{BR}, who give a necessary and
sufficient set of relations modulo which a Tutte polynomial of a
colored graph is labeling independent. It turns out that the relative Tutte polynomial defined in this paper also possesses this property. That is, it can also be computed by a process similar to the spanning tree expansion by counting the activities of the edges of $G$ (that are not in $\H$)
with respect to a given order of the edges of $G$ (that are not in
$\H$). Furthermore, the polynomial so defined is also independent
of the order on the edges. It is worthwhile to point out that
since our polynomial is defined on colored graphs, it also
generalizes the {\em set-pointed} Tutte polynomial introduced and
discussed in \cite{Ver} under the graph theoretical setting.
Our approach follows closely the one used in Bollob\'as
and Riordan \cite{BR} and our result is analogous to that
of \cite{BR}. In other words, our result is also the most general
in the sense that we have given a necessary and sufficient set of
relations modulo which a relative Tutte polynomial of a colored
graph is labeling independent.

Our main motivation to introduce and study the relative Tutte polynomial
comes from knot theory. It is well-known that the Jones polynomial of a
link can be computed from the Kauffman bracket polynomial. On the other
hand, the Kauffman bracket polynomial of a link can be computed from the
(signed) Tutte polynomial of the face graph of a regular
projection of the link. This was first shown for alternating links and
the ordinary Tutte polynomial by Thistlethwaite~\cite{T0}, then
generalized to arbitrary links and a signed Tutte polynomial by
Kauffman~\cite{K2}. This enables applications of the
ordinary Tutte polynomials and their signed generalizations to
classical knot theory such as those in \cite{DEZ,DGH1,Ja}.
For virtual knots the situation is a little more complicated.
An appropriate generalization of the Kauffman bracket polynomial
was developed by Kauffman himself~\cite{K3}. However, until very
recently, no appropriate generalization of the Tutte polynomial to face
graphs of virtual links was known. In a series of papers, Chmutov,
Pak and Voltz \cite{Ch,CP,CV} developed a generalization of
Thistlethwaite's theorem first to checkerboard-colorable \cite{CP} then
to arbitrary \cite{Ch,CV} virtual link diagrams. These express the
Jones polynomial of a virtual link in terms of a signed generalization
of the Bollob\'as-Riordan polynomial~\cite{BR2,BR3} of a ribbon graph,
obtained from the virtual link diagram. In this paper we will show that
a relative variant of the {\em other} generalization of the Tutte
polynomial, also due to Bollob\'as and Riordan~\cite{BR} may also be
used to compute the Jones polynomial of a virtual link, this time
directly from the face graph of the virtual link diagram.
The application of the relative Tutte polynomials is not just
limited to virtual knot theory. We remind the reader that the
Bollob\'as-Riordan polynomial defined in~\cite{BR} is in a sense
the most general Tutte polynomial that may be defined for colored
graphs~\cite[Theorem 2]{BR}. Several examples of less general Tutte
polynomials arising as a homomorphic image of this colored Tutte polynomial
are given in~\cite{BR}, and the relative variant of the colored
Bollob\'as-Riordan Tutte polynomial may also be used in the study of all
models where the original, non-relative variant proved itself useful.
For example, it can be applied to
networks with different layers of structures as well, generalizing some
results in \cite{FK}.

This paper is organized in the following way. In the Preliminaries we review the main result of Bollob\'as and
Riordan~\cite{BR}, providing in a sense the most general notion of
a colored Tutte polynomial that is labeling independent if we generalize
Tutte's original approach~\cite{Tu} of counting activities. We
then turn to the introduction and discussions of the relative
Tutte polynomial of a colored graph in Section \ref{s3}. There we
state and prove our main theorem about the relative Tutte
polynomial for connected graphs and matroids. In Section \ref{s4}, we extend the relative Tutte polynomial to disconnected graphs.  We also give some examples of the relative Tutte polynomials in this section. In one example, we show how an ordinary colored Tutte polynomial can be recovered from a relative Tutte polynomial. In Section \ref{s5}, we apply the relative Tutte polynomial to virtual knot theory. There we will state and prove our main theorem in the application of the relative Tutte polynomial.

\bigskip
\section{Preliminaries}\label{s2}

\subsection{Matroids associated to graphs}

The Tutte polynomial and its colored generalizations are {\em matroid
invariants}. There are many equivalent definitions of a matroid, a good
basic reference is \cite{We}. A matroid is {\em graphic} if its
elements may be represented by the edges of a graph $G$ such that the
minimal dependent sets are the cycles in the graph. Independent sets
correspond then to forests, maximal independent sets or bases to
spanning forests. If the graph is connected then the bases are the
spanning trees. Since the matroid associated to a graph depends only on
the cycle structure, two graphs have the same underlying matroid
structure if they have the same $2$-edge connected components.
The Tutte polynomial and the generalizations we consider are matroid
invariants in the sense that they depend only on the matroid associated
to the graph. Most statements and proofs we make or were made about
(generalized) Tutte polynomials of graphs may be easily generalized to
matroids.

In this paper we will often rely on the notion of {\em matroid
  duality}.  Perhaps the easiest way to define the dual ${\mathcal M}^*$ of
a matroid ${\mathcal M}$ is by giving the maximal independent sets of
${\mathcal M}^*$: the set $B^*$ is a basis in ${\mathcal M}^*$ if and
only if its complement is a basis in ${\mathcal M}$. It is well-known
that the deletion and contraction operations are duals of each
  other. The dual of  a graphic matroid is not necessarily graphic,
  in fact, only planar graphs have a dual graph. However, the dual
  notion of a cycle in a graph is well known. Assume that a matroid is
  represented as the cycle matroid of a connected graph. Then a set is a
  cocycle (=cycle in the dual) if and only if it is a minimal cut,
  i.e., a minimal disconnecting set. If we make a pair of dual statements
  about a graph, it is sufficient to prove only one of the two
  statements if that proof generalizes immediately to matroids. The
  proof of the dual statement may be obtained by replacing each notion
  with its dual in the proof of the original statement.

\subsection{Tutte polynomials}

In this section, we review the results of Bollob\'as and Riordan~\cite{BR} concerning the Tutte polynomial for a colored
connected graph, as well as some results we had obtained in our earlier work \cite{DGH2}. A graph $G$
with vertex set $V$ and edge set $E$ is a colored graph if every
edge of $G$ is assigned a value from a color set $\Lambda$. The
following notion of ``activities'' was first introduced by
Tutte~\cite{Tu} for non-colored graphs to express the ordinary
Tutte polynomial as a sum of contributions over all spanning trees
of a connected graph.

\begin{definition}
Let $G$ be a connected graph with edges labeled $1, 2, \ldots,
n$, and let $T$ be a spanning tree of $G$. An edge $e$ of $T$ is
said to be internally active if for any edge $f\not=e$ in G such
that $(T \setminus  e ) \cup  f $ is a spanning tree of $G$, the
label of $e$ is less than the label of $f$. Otherwise $e$ is said
to be internally inactive. On the other hand, an edge $f$ of $G
\setminus T$ is said to be externally active if $f$ has the
smallest label among the edges in the unique cycle contained in $T
\cup   f  $. Otherwise, $f$ is said to be externally inactive.
\end{definition}

\vspace{-0.4cm} Bollob\'as and Riordan~\cite{BR} use Tutte's
notion of activities but generalize Tutte's variable assignments
as follows. Let $G$ be a colored and connected graph and $T$ a
spanning tree of $G$. For each edge $e$ in $G$ with color
$\lambda$, we assign one of the variables $X_\lambda$,
$Y_\lambda$, $x_\lambda$ and $y_\lambda$ to it according to the
activities of $e$ as shown below (with respect to the tree $T$):

\begin{table}[h]\label{table1}
\begin{center}
\begin{tabular}{|c|c|c|c|}
\hline internally active & $X_\lambda$& externally active &
$Y_\lambda$\\
\hline
internally inactive & $x_\lambda$& externally inactive & $y_\lambda$\\
\hline
\end{tabular}
\end{center}
\caption{The variable assignment of an edge with respect to a
spanning tree $T$.}
\end{table}

\begin{definition} Let $G$ be a connected colored graph. For a
spanning tree $T$ of $G$, let $C(T)$ be the product of the
variable contributions from each edge of $G$ according to the
variable assignment above, then the Tutte polynomial $T(G)$ is
defined as the sum of all the $C(T)$'s over all possible spanning
trees of $G$.
\end{definition}

Tutte's original variable assignment may be recovered by setting
all $X_{\lambda}=x$, $Y_{\lambda}=y$, $x_{\lambda}=1$ and
$y_{\lambda}=1$ for all $\lambda\in\Lambda$. It is Tutte's main
result, that the total contribution of all spanning trees is
labeling independent in the non-colored case. This property does
not generalize to the colored case. To remedy the situation, in
the definition of most colored Tutte polynomials in the literature
one needs to factor the polynomial ring ${\mathbb
Z}[\Lambda]:={\mathbb
Z}[X_\lambda,Y_\lambda,x_\lambda,y_\lambda:\lambda\in \Lambda]$
with an appropriate ideal $I$, such that the formula for $T(G)$ in
${\mathbb Z}[\Lambda]/I$ becomes labeling independent. An exact
description of all such ideals was given by Bollob\'as and
Riordan~\cite[Theorem 2]{BR}.

\begin{proposition}[Bollob\'as-Riordan]
\label{T_BR} Assume $I$ is an ideal of ${\mathbb Z}[\Lambda]$.
Then the homomorphic image of $T(G)$ in ${\mathbb Z}[\Lambda]/I$
is independent of the labeling of the edges of $G$ if and only if
\begin{equation}
\det\left(\begin{array}{ll}
X_\lambda & y_\lambda\\
X_\mu & y_\mu
\end{array}\right)
- \det\left(\begin{array}{ll}
x_\lambda & Y_\lambda\\
x_\mu & Y_\mu
\end{array}\right)\in I,
\end{equation}
\begin{equation}
Y_{\nu}\det\left(\begin{array}{ll}
x_\lambda & Y_\lambda\\
x_\mu & Y_\mu
\end{array}\right)
- Y_{\nu}\det\left(\begin{array}{ll}
x_\lambda & y_\lambda\\
x_\mu & y_\mu
\end{array}\right)\in I,
\end{equation}
and
\begin{equation}
X_{\nu}\det\left(\begin{array}{ll}
x_\lambda & Y_\lambda\\
x_\mu & Y_\mu
\end{array}\right)
- X_{\nu}\det\left(\begin{array}{ll}
x_\lambda & y_\lambda\\
x_\mu & y_\mu
\end{array}\right)\in I.
\end{equation}
hold for all $\lambda,\mu, \nu \in \Lambda$.
\end{proposition}

Bollob\'as and Riordan denote the ideal generated by the
differences listed in Proposition~\ref{T_BR} by $I_0$. The homomorphic
image of $T(G)$ in ${\mathbb Z}[\Lambda]/I_0$ is the most general
colored Tutte polynomial whose definition is independent of the
labeling. Many important polynomials may be obtained from this
most general colored Tutte polynomial by substitution, and most
such substitutions map ${\mathbb Z}[\Lambda]/I_0$ into an integral
domain in such a way that the image of the variables $x_{\lambda},
X_{\lambda}, y_{\lambda}$ and $Y_{\lambda}$ is nonzero. As it is
implicitly noted in~\cite[Corollary 3]{BR}, all such substitutions
factor through the canonical map ${\mathbb
Z}[\Lambda]/I_0\rightarrow {\mathbb
  Z}[\Lambda]/I_1$ where
$I_1$ is the ideal generated by all polynomials of the form
\begin{equation}
\label{E_I11} \det\left(\begin{array}{ll}
X_{\lambda}& y_{\lambda}\\
X_{\mu}& y_{\mu}
\end{array}\right)
- \det\left(\begin{array}{ll}
x_{\lambda}& y_{\lambda}\\
x_{\mu}& y_{\mu}
\end{array}\right)
\end{equation}
and
\begin{equation}
\label{E_I12} \det\left(\begin{array}{ll}
x_{\lambda}& y_{\lambda}\\
x_{\mu}& y_{\mu}
\end{array}\right)
- \det\left(\begin{array}{ll}
x_{\lambda}& Y_{\lambda}\\
x_{\mu}& Y_{\mu}
\end{array}\right).
\end{equation}
Since $I_1$ properly contains $I_0$, the canonical image of the
Tutte polynomial in ${\mathbb Z}[\Lambda]/I_1$ is
labeling independent. Moreover, we highlight the following
algebraic observation, making \cite[Corollary 3]{BR} truly useful.

\begin{lemma}\cite{DGH2}
\label{L_ip} The ideal $I_1$ is a prime ideal. More generally,
given any integral domain $\R$, the ideal $I_1$ generated by all
elements of the form (\ref{E_I11}) and (\ref{E_I12}) in
$\R[\Lambda]$ is prime.
\end{lemma}

As noted in~\cite[Remark 3]{BR} and in \cite{DGH2}, the above
definitions and statements may be generalized to {\em matroids} without
essential adjustment. In fact, the definitions of activities may be restated by
replacing the word ``spanning tree'' with ``matroid basis'' and
interpreting the word ``cycle'' as a ``minimal dependent set''.
In particular there is a trivial generalization of our Tutte
polynomial to {\em disconnected graphs}, by replacing the word
``spanning tree'' with {\em spanning forest}. Given a disconnected graph $G$
with connected components $G_1,\ldots,G_k$, the Tutte polynomial
$T(G)$ obtained via this generalization is the product of the
Tutte polynomials of its components.
There are also other generalizations to disconnected graphs which
keep track of the number of connected components. Such
generalizations are discussed in~\cite[Section 3.4]{BR}. In order
to be able to represent a larger class of graph-theoretic
polynomials by substitution into the variables of their colored
Tutte polynomial, Bollob\'as and Riordan~\cite{BR} introduced a
multiplicative constant $\alpha_{k(G)}$ (depending only on
the number $k(G)$ of connected components in $G$) to $T(G)$. Consequently
they generalized their result~\cite[Theorem 2]{BR}, providing a
necessary and sufficient condition to have a labeling independent
Tutte polynomial $\alpha_{k(G)}T(G)$. We wish to stress that these
disconnected generalizations depend on the number of connected
components which {\em cannot be recovered from the
cycle matroid of the graph}.

\medskip
Finally, since the Tutte polynomials considered above are labeling
independent, we have the following recursive formula
\begin{equation}\label{recur}
T(G)=\left\{
\begin{array}{ll}
y_\lambda T(G\setminus e)+x_\lambda T(G/e)&\quad e\ {\rm is}\ {\rm
neither}\ {\rm a}\ {\rm loop}\ {\rm nor}\ {\rm a}\ {\rm bridge},\\
Y_\lambda T(G\setminus e)&\quad e\ {\rm is}\ {\rm a}\ {\rm loop},\\
X_\lambda T(G/e)&\quad e\ {\rm is}\ {\rm a}\ {\rm bridge},
\end{array}
\right.
\end{equation}
where $\lambda$ is the color of $e$, $G\setminus e$ is the graph
obtained from $G$ by deleting $e$ and $G/e$ is the graph obtained
from $G$ by contracting $e$. (See \cite[(3.14)]{BR} and the
Preliminaries in~\cite{DGH2}.)

\bigskip
\section{Relative Tutte Polynomials}\label{s3}

In this section we introduce the relative Tutte polynomial of a
connected colored
graph $G$, with respect to a set of edges $\H\subset E(G)$ (where $E(G)$ is the edge set of $G$) and prove our generalization
of Theorem \ref{T_BR} of Bollob\'as and Riordan \cite{BR} to our
relative Tutte polynomials. The definitions and the main result may be
easily generalized to colored matroids, by keeping the relative Tutte
polynomial $\psi(G)$ of a graph satisfying $\H=E(G)$ a matroid
invariant. However, as we will see below, a larger class of mappings
$\psi$ fits our theory. In Section~\ref{s4}  these other
generalizations will allow us to
consider the Tutte polynomials of disconnected graphs introduced by
Bollob\'as and Riordan~\cite[Section 3.4]{BR} as a special instance of our
relative Tutte polynomials, even without extending our definitions to
disconnected graphs. In our presentation we will focus on graphs
and indicate along the way in remarks how the immediate generalization
to matroids  may be made, when applicable. However, each time we omit half of
the proof of a lemma, it is implied that the statement should be about
matroid and what is left to prove is the dual of what we have already
shown, using only the matroid structure.

\begin{definition}\label{TD1}
Let $G$ be a connected graph and $\H$ a subset of its edge set
$E(G)$. A subset $\C$ of the edge set $E(G)\setminus \H$ is called a
{\em contracting set} of $G$ with respect to $\H$ if $\C$ contains no
cycles and $\D:=E(G)\setminus(\C\cup \H)$ contains no cocycles (and $\D$ is called a {\em deleting set}).
\end{definition}
Note that the sets $\C$ and $\D$ mutually determine each other, by their
disjoint union being $E(G)\setminus \H$.

\begin{remark}
{\em Definition~\ref{TD1} may be used without any change to define
a contracting set $\C$ and  a deleting set $\D$ for any matroid with
respect to a set of elements $\H$. For matroids we obtain
a ``self-dual'' dual notion:
$(\C,\D)$ is a pair of contracting and deleting sets for a matroid
${\mathcal M}$ if and only of $(\D,\C)$ is a pair of contracting and
deleting sets for the dual matroid ${\mathcal M}^*$. }
\end{remark}

\begin{remark}\label{R_nod}
{\em Recalling that the cocycles in a
connected graph $G$ are exactly the minimal cuts, $\D$ contains no
cocycles if and only if the deletion of $\D$ does not disconnect the
graph $G$.}
\end{remark}

Sometimes we will refer to $\C$, $\D$ and
$\H$ as {\em graphs}, and in each such instance we consider them as the subgraphs of $G$ induced by the respective set of edges.
Since the spanning trees of $G$ are the cycle-free connected subgraphs of $G$, at the light of Remark~\ref{R_nod} the following lemma is obvious.

\begin{lemma}\label{TL1}
In the above definition, if $\H=\emptyset$, then $\C\subseteq E(G)$ is a
contracting set if and only if the subgraph $\C$ is a
spanning tree of $E(G)$.
\end{lemma}

To prove the generalizations of our statements for matroids, the following
observation will be useful.
\begin{lemma}
\label{L_b}
$\C$ is a contracting set with respect to $\H$ if and only if there is a
basis $B\subset \C\cup\H$ that contains $\C$.
\end{lemma}
\begin{proof}
Assume first $\C$ is a contracting set. Then $\D:=E(G)\setminus (\C\cup \H)$
contains no cocycle, so it is co-independent, contained in a dual
basis. The complement of this dual basis is a basis contained in
$\C\cup\H$. Thus $\C\cup \H$ has full rank. Since $\C$ contains no
cycle, it is independent. Extending the independent set $\C$ to a
maximal independent subset of $\C\cup\H$ yields a basis
containing $\C$ and contained in $\C\cup\H$.

Assume now that there is a basis $B\subset \C\cup\H$ that contains $\C$. Then $\C\subset B$ is independent and contains no cycle. On
the other hand, $E(G)\setminus B$ is a dual basis containing $\D$. Hence
$\D$ is co-independent and contains no cocycle.
\end{proof}

Note that even if $E(G)\setminus \H$ contains edges that are not
loops, it is possible that $\C=\emptyset$ (as the following example
shows). This is not the case when $\H=\emptyset$.
On the other hand, it is important to note that for any edge $e\in
\D$, the graph $\{e\}\cup {\C\cup \H}$ must contain a cycle with
$e$ in it (although such a cycle may not be unique), for otherwise
$e$ would be a cut edge of $\{e\}\cup {\C\cup \H}$, contradicting
the fact that deleting $\D$ cannot disconnect the graph.

\medskip
\begin{example}
In the following figure, $\H=\{e_2\}$. So by definition, $\{e_1\}$
can serve either as $\C$ or $\D$.
\begin{figure}[!htb]
\begin{center}
\includegraphics[scale=.8]{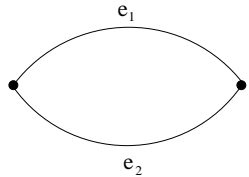}
\caption{\label{RTuttefig1} A graph $G$ of two edges with $\H$
containing one edge.}
\end{center}
\end{figure}
\end{example}

\begin{lemma}\label{TL3}
Let $G$ be a connected graph and $\H$ a subset of $E(G)$. Let $\C$
be a contracting set of $G$ with respect to $\H$, $\D$ be the
corresponding deleting set and $e\in \C$ be any edge in $\C$. Then
for any $f\in \D$, $\C^\p=\{f\}\cup (\C\setminus \{e\})$ is also a
contracting set with respect to $\H$ if the triplet $(\C,e,f)$
has either of the following properties:
\begin{itemize}
\item[(i)] $\C\cup \{f\}$ contains a cycle containing $\{e\}$.
\item[(ii)] $\D\cup \{e\}$ contains a cocycle containing
  $\{f\}$.
\end{itemize}
Moreover, if the triplet $(\C,e,f)$ satisfies (i) or (ii) then the
triplet $(\C^\p,f,e)$ has the same properties.
\end{lemma}
\begin{proof}
Assume first that (i) holds. Since $\C$ contains no cycle, adding $f$
creates at most one cycle $C_0$. By our assumption, $e\in C_0$. By
Lemma~\ref{L_b} there is a basis $B$ containing $\C$, contained in
$\C\cup \H$. Adding $f$ to $B$ creates only one cycle, so this cycle is
still $C_0$. Thus $B^\p:=B\cup \{f\}\setminus \{e\}$ contains no cycle,
and it is a basis. Clearly $B^\p$ is a basis containing $\C^\p$ and
contained in $\C^\p\cup\H$. Thus $\C^\p$ is a contracting set by
Lemma~\ref{L_b}. Furthermore $\C^\p\cup\{e\}$ contains the cycle $C_0$
which contains $f$.

The statements about property (ii) follow by ``dualizing'' the above
argument.
\end{proof}
\begin{remark}
{\em The statement and proof of Lemma~\ref{TL3} may be generalized to
  matroids immediately. For graphs, property (ii) is equivalent to the
  following:

(ii') $e$ is a bridge in $\C\cup \H$ and $(\C\cup \H\cup \{f\})\setminus \{e\}$ is connected.
}
\end{remark}
It should be noted that properties (i) and (ii) are not necessarily mutually
exclusive for a triplet $(\C,e,f)$. As seen in the proof of
Lemma~\ref{TL3},  if a triplet $(\C,e,f)$
has property (i) then the cycle $C_0$
contained in $\C\cup \{f\}$ is unique, and
equal to the unique cycle contained in $\C^\p \cup \{e\}$. If a triplet
$(\C,e,f)$ has property (ii) but does not have property (i) then every
cycle contained in $\C\cup \H\cup\{f\}$ and containing $\{e,f\}$ contains at
least one element of $\H$, and there is at least one such cycle.

In the next definition, we define the relative activities of  edges
of $E(G)\setminus \H$ with respect to a contracting set $\C$ of $G$
based on a particular labeling of the edges  of $G$.

\begin{definition}{\em
Let $G$ be a connected graph and $\H$ be a subset of $E(G)$. Let us
assume that a labeling  of $G$ is given in such a way that all
edges in $\H$ are labeled with number $0$ and all other edges are
labeled with distinct positive integers. Such a labeling is called
a {\em proper labeling} or a {\em relative labeling} (with respect
to $\H$). In other words, a proper labeling of the edges of $G$ with
respect to $\H$ is a map $\phi:\ E(G)\longrightarrow
\mathbb{Z}$ such that $\phi(e)=0$ for any $e\in \H$ and $\phi$
is an injective map from $E(G)\setminus \H$ to $\mathbb{Z}^+$.  We
say that $e_1$ is larger than $e_2$ if $\phi(e_1)>\phi(e_2)$. Let $\C$ be a contracting set of $G$
with respect to $\H$, then

a) an edge $e\in \C$ is called {\em internally active} if $\D\cup \{e\}$
contains a cocycle $D_0$ in which $e$ is the smallest edge,
otherwise it is {\em internally inactive}.

b) an edge $f\in \D$ is called {\em externally active} if $\C\cup \{f\}$
contains a cycle $C_0$ in which $f$ is the smallest edge,
otherwise it is {\em externally inactive}.
}\end{definition}

\begin{remark}\label{r-active}{\em
For any $f\in\D$, if $\C\cup\{f\}$ contains a cycle then this cycle is unique. For a fixed $\C$ we may identify the internally or
externally active edges by comparing the label of each $e\in \C$ with
the label of each $f\in \D$ such that the triplet $(\C,e,f)$ has at
least one of the properties considered in Lemma \ref{TL3}.}
\end{remark}

\begin{remark}\label{r-active-d}{\em
An internal edge $e\in\C$ is active only if it closes a cocycle
in $\D$. For a connected graph, $\D\cup\{e\}$ contains a cocycle if and
only if removing $\D\cup\{e\}$ disconnects the graph. This is equivalent
  to stating that $e$ is not contained in any cycle of
$\C\cup \H$ or, that $e$ is
a bridge in $\C\cup \H$. This rephrasing remains true for matroids in
general if we define ``bridge'' as ``coloop''. (The proof is left to the
reader.)  A bridge $e$ in $\C\cup \H$ is active exactly
when it is smaller than
any $f\in\D$ connecting the two components of $\C\cup\H\setminus
\{e\}$.  Hence we may restate the condition for internal activity as
follows: an edge $e\in\C$ is internally active
if whenever
an edge $f\not\in \C$ closes a cycle in $\{f\}\cup \C\cup \H$
containing $e$, $e$ is always smaller than
$f$; otherwise $e$ is said to be {\em internally inactive}. This last
rephrasing holds again for matroids in general.  In
particular, if $e$ is on a cycle containing only edges from $\C\cup
\H$, then $e$ is internally inactive.
}
\end{remark}

\begin{remark}\label{equi-act-def}
{\em
The following statements offer an equivalent definition of the
activity of a regular edge of $G$. An edge $e\in \C$ is internally
active if and only if it becomes a bridge once all edges in $\D$
larger than $e$ are deleted. On the other hand,
an edge $f\in \D$ is externally active if and only if it becomes a
loop after all edges in $\C$ larger than $f$ are
contracted.}
\end{remark}

Before we give a formal definition of our relative Tutte polynomial
of a colored graph $G$ with respect to a subset $\H$ of edges of
$G$, let us explain the role we intend the set $\H$ to play. In the
classical case where $\H=\emptyset$, when we apply the recursive
formula \ref{recur} to compute $T(G)$, every edge of $G$ is either
deleted or contracted in the process. At the end, only one vertex is
left. Naturally, the Tutte polynomial of a vertex is defined to be
$1$. In our case, we would like to be able to compute our relative
Tutte polynomial by the same recursive rule (\ref{recur}), however we
want to preserve the special zero edges (edges in $\H$) in this
process, as these can be special edges that may not allow the use of a
contraction/deletion formula similar to \ref{recur} (which is the case
in our application to virtual knot theory). By doing so, at the end of
the process, we will end up with graphs $\H_{\C}$
obtained from $G$ by contracting the edges in $\C$ and deleting the
edges in $\D$. We have the option of defining a weight $\psi(\H_{\C})$ for these graphs in a different manner. In the proof of our main result we will only need to be
able to guarantee that for any triplet $(\C,e,f)$ satisfying property
(i) or (ii) in Lemma~\ref{TL3}, we associate the same value to the
graphs $\H_\C$ and $\H_{\C^\p}$. To guarantee this, we require the
mapping $\psi$ to have the following property.
\begin{definition}
\label{Dpsi}
Let $\psi$ be a mapping defined on the
isomorphism classes of finite connected graphs with values in a ring
$\R$. We say that $\psi$ is a {\em block invariant}
if for all positive integer $n$ there is a function $f_n: \R^n\rightarrow
\R$ that is symmetric under permuting its input variables such that for
any connected graph $G$ having $n$ blocks $G_1$, \ldots, $G_n$ we have
$$
\psi(G)=f_n(\psi(G_1),\ldots,\psi(G_n)).
$$
\end{definition}
In other words, we require the ability to compute $\psi(G)$ from
the value of $\psi$ on the blocks of $G$, and
this computation should not depend on the order in which the blocks
are listed.
\begin{lemma}
\label{HC}
Let $G$ be a connected graph and $\H$ be a subset of $E(G)$.
Assume that $\C$ is a contracting set with respect to $\H$ and
that the triplet $(\C,e,f)$ has at least one of the properties listed
in Lemma~\ref{TL3}. Let $\C^\p:=(\C\cup\{f\})\setminus \{e\}$.
Then the multiset of blocks of $\H_{\C}$ is the same as the multiset of
blocks of $\H_{\C^\p}$.
\end{lemma}
\begin{proof}
The statement is obvious when $(\C,e,f)$ has property (i), since
then $\H_\C$ and $\H_{\C^\p}$ are isomorphic as graphs.
If $(\C,e,f)$ has property (ii), then $\H_{\C^\p}$ is obtained from
  $\H_\C$ by expanding the edge $e$ (which becomes a bridge), removing
$e$, adding a new bridge $f$, and contracting $f$. Clearly $\H_\C$ and
  $\H_{\C^\p}$  have the same blocks.
\end{proof}
Recall that the cycle matroid of a graph depends only on its $2$-edge
connected components which are subsets of its blocks. Thus a matroid
invariant of connected graphs is also a block invariant.
When we generalize our notion of the relative Tutte polynomial to
matroids we want to require $\psi$ to be a matroid invariant with values
in a fixed integral domain $\R$. Then we need the following variant of
Lemma~\ref{HC}:

\begin{lemma}
\label{HCm}
Let $M$ be a matroid and $\H$ a subset of its elements.
Assume that $\C$ is a contracting set with respect to $\H$ and
that the triplet $(\C,e,f)$ has at least one of the properties listed
in Lemma~\ref{TL3}. Let $\C^\p:=(\C\cup\{f\})\setminus \{e\}$.
Then the cycle matroid of the graph $\H_{\C}$ is the same as the cycle
matroid of the graph $\H_{\C^\p}$.
\end{lemma}
\begin{proof}
Assume that $(\C,e,f)$ has property (i).
A set $X\subseteq \H$ is independent in $\H_{\C}$ if and only if there
is a basis $B$ of the original graph containing $X\cup\C$. As seen in
the proof of Lemma~\ref{TL3}, the set $B^\p:=(B\cup\{f\})\setminus
\{e\}$ is also a basis of the original graph, and this basis contains
$X\cup \C^\p$. Thus if $X\subseteq \H$ is independent in $\H_{\C}$, it
is also independent in $\H_{\C^\p}$. The converse is also true since, by
Lemma ~\ref{TL3}, the triplet $(\C^\p,f,e)$ also has property (i).

The proof for the case when $(\C,e,f)$ has property (ii) may be obtained
by ``dualizing'' the above argument.
\end{proof}

Let $G$ be a connected graph and $\H\subseteq E(G)$. Assume
we are given a mapping $c$ from $E(G)\setminus \H$ to a color set
$\Lambda$.  Assume further that $\psi$ is a block invariant
associating an element of a fixed integral domain $\R$ to each connected
graph. For any contracting set  $\C$ of $G$ with
respect to $\H$, let $\H_\C$ be the graph obtained by deleting all
edges in $\D$ and contracting all edges in $\C$ (so that the only
edges left in $\H_\C$ are the zero edges). Finally, we will assign a
proper labeling to the edges of $G$. We now define the relative
Tutte polynomial of $G$ with respect to $\H$ and $\psi$ as
\begin{equation}\label{eq1}
T_\H^\psi(G)=\sum_{\C} \big(\prod_{e\in G\setminus
H}w(G,c,\phi,\C,e)\big)\psi(\H_\C)\in \R[\Lambda],
\end{equation}
where the summation is taken over all contracting sets $\C$
and $w(G,c,\phi,\C,e)$  is the {\em weight} of the edge $e$ with
respect to the contracting set $\C$, which is defined as
(assume that $e$ has color $\lambda$):

\begin{equation}
w(G,c,\phi,\C,e)=\left\{
\begin{array}{ll}
X_\lambda &\ {\rm if}\ e\ {\rm is\ internally\ active};\\
Y_\lambda &\ {\rm if}\ e\ {\rm is\ externally\ active};\\
x_\lambda &\ {\rm if}\ e\ {\rm is\ internally\ inactive};\\
y_\lambda &\ {\rm if}\ e\ {\rm is\ externally\ inactive}.
\end{array}
\right.
\end{equation}

To simplify the notation somewhat, we will be using $T_\H(G)$ for
$T_\H^\psi(G)$, with the understanding that some $\psi$ has been
chosen, unless there is a need to stress what $\psi$ really is.
Following \cite{BR}, we then write
$$
W(G,c,\phi,\C)=\prod_{e\in G\setminus \H}w(G,c,\phi,\C,e)
$$
so that
\begin{equation}\label{eq2}
T_\H(G,\phi)=\sum_{\C} W(G,c,\phi,\C)\psi(\H_\C).
\end{equation}

We are now able to extend Theorem \ref{T_BR} of Bollob\'as and Riordan
\cite{BR} to $T_\H$.

\begin{theorem}\label{DH1}
Assume $I$ is an ideal of $\R[\Lambda]$. Then the
homomorphic image of $T_\H(G,\phi)$ in $\R[\Lambda]/I$ is
independent of $\phi$ (for any $G$ and $\psi$) if and only if
\begin{equation}\label{Ieq1}
\det\left(\begin{array}{ll}
X_\lambda & y_\lambda\\
X_\mu & y_\mu
\end{array}\right)
- \det\left(\begin{array}{ll}
x_\lambda & Y_\lambda\\
x_\mu & Y_\mu
\end{array}\right)\in I
\end{equation}
and
\begin{equation}\label{Ieq2}
\det\left(\begin{array}{ll}
x_\lambda & Y_\lambda\\
x_\mu & Y_\mu
\end{array}\right)
- \det\left(\begin{array}{ll}
x_\lambda & y_\lambda\\
x_\mu & y_\mu
\end{array}\right)\in I.
\end{equation}
hold for all $\lambda,\mu \in \Lambda$.
\end{theorem}

\begin{proof}
We will follow the approach and notation of \cite{BR} as closely as
possible, so the reader  familiar with that paper can follow
easily. In this approach, for the sufficient condition, it suffices
to show that the relative Tutte polynomial defined in equation
(\ref{eq1}) is the same for any two proper labelings of $G$ that
differ only by a transposition under conditions (\ref{Ieq1}) and
(\ref{Ieq2}). That is, if  two proper labelings $\phi$ and $\phi^\p$
differ only on two regular edges $e$ and $f$ such that
$\phi(f)-\phi(e)= 1$ and $\phi^\p(e)=\phi(f)$, $\phi^\p(f)=\phi(e)$,
then the relative Tutte polynomial defined in equation (\ref{eq1})
is the same for $\phi$ and $\phi^\p$ under conditions (\ref{Ieq1})
and (\ref{Ieq2}). Without loss of generality let us assume that
$\phi(e)=\phi^\p(f)=i$ and $\phi(f)=\phi^\p(e)=i+1$.

We now proceed to prove this fact. Let $\C$ be any contracting set
of $G$ with respect to $\H$ and $\D=G\setminus (\C\cup \H)$.
Clearly, by the definition of activities, if both edges $e$ and $f$
are in $\C$ or in $\D$, then they will never be compared to each
other with respect to their activities, therefore their label
comparison to any other regular edge stays the same under the two
labelings. This is even more obvious if one uses the equivalent
definition of activities in Remark \ref{equi-act-def}. Thus, in that
case, all regular edges will make the same contributions under the two
labelings. Furthermore, as noted at the end of Remark~\ref{r-active},
the labels of $e$ and $f$ are still not compared to each other
unless the triplet $(\C,e,f)$ or the triplet $(\C,f,e)$ has at least one
of the properties listed in Lemma~\ref{TL3}.
Without loss of generality we may assume
 that $e\in \C$, $f\in \D$, and the triplet satisfies at least one of
 the properties listed in Lemma~\ref{TL3}. By Lemma~\ref{TL3}, the set
$\C^{\p}:=\C\setminus \{e\}\cup\{f\}$ is also a contracting set, and
the triplet $(\C^\p, f,e)$ has the same properties as $(\C,e,f)$.
As noted in the proof of Theorem \ref{T_BR}  in \cite{BR} in a similar
situation, it suffices to show that for such pairs $(\C,\C^\p)$ the total
contribution of the two contracting sets is the same with respect to
$\phi$ and $\phi^{\p}$:
\begin{equation}\label{eq3}
W(G,c,\phi,C)+W(G,c,\phi,C^\p)-W(G,c,\phi^\p,C)-
W(G,c,\phi^\p,C^\p)\in I.
\end{equation}
There are three cases, depending on whether the triplet $(\C,e,f)$ has
only one or both properties listed in Lemma~\ref{TL3}. Let
$\lambda=c(e)$, $\mu=c(f)$.

{\bf Case 1.} $(\C,e,f)$ has property (ii) but not (i). In this case
$\C\cup \{f\}$ contains no cycle, since $\C$ is cycle free, $f$ is a
bridge in $(\C\cup\H\cup\{f\})\setminus \{e\}$, so a cycle contained
in $\C\cup \{f\}$ contains $f$, and a cycle containing $f$ in
$\C\cup\H\cup\{f\}$ contains $e$. We obtained that a cycle contained
in $\C\cup \{f\}$ must contain $e$, but we assume property (i) to be
false. Therefore $f$ is externally inactive with respect to $\C$, under
both labelings. Similarly, $\C^\p\cup\{e\}$ does not contain any cycle,
and so $e$ is externally inactive with respect to $\C^\p$, under
both labelings. Since we assume $\phi(e)<\phi(f)$, $f$ is internally
inactive with respect to $\C^\p$, under $\phi$. Similarly
$\phi^\p(f)<\phi^\p(e)$ implies that $e$ is internally
inactive with respect to $\C$, under $\phi^\p$. So far we obtained that $e$ and
$f$ contribute the following factors to the polynomials  $W(G,c,\phi,C)$,
$W(G,c,\phi,C^\p)$, $W(G,c,\phi^\p,C)$, $W(G,c,\phi^\p,C^\p)$:
$$
\begin{array}{c|cc}
&\C&\C^\p\\
\hline
\phi& (?)y_{\mu}& x_{\mu} y_{\lambda}\\
\phi^\p& x_{\lambda}y_{\mu}& (?)y_{\lambda}\\
\end{array}
$$
The question marks indicate that we are left to determine whether $e$ is
internally active with respect to $\C$ under $\phi$ and whether $f$ is
internally active with respect to $\C^\p$ under $\phi^\p$. The
information at hand does not allow to determine whether these edges are
active or not in the given context. However, we can show that the answer
is either simultaneously ``yes'' or simultaneously ``no'' to both
questions. In fact, $e$ is internally active with respect to $\C$ under
$\phi$ if and only if all edges of $\D$  that close a cycle containing
$e$ with $\C\cup\H\setminus\{e\}=\C^\p\cup\H\setminus\{f\}$
have larger label than the label of $e$. (This question has the same
answer under $\phi$ and $\phi^\p$.) Similarly $f$ is internally active
with respect to $\C^\p$ under $\phi^\p$ if and only if all edges of $\D$
that close a cycle containing $f$ with
$\C^\p\cup\H\setminus\{f\}=\C\cup\H\setminus\{e\}$ have larger label
than the label of $e$. We are comparing the labels of the same set of
deleting edges to the adjacent labels of $e$ and $f$.

If the answer to both remaining questions is ``no'', then the left hand
side of (\ref{eq3}) becomes zero, otherwise (\ref{eq3}) becomes
\begin{equation}
\label{Ecase1}
W_0(X_\lambda y_\mu +x_\mu y_\lambda - x_\lambda y_\mu -X_\mu
y_\lambda)\in I.
\end{equation}
Here $W_0$ is the product of
$\psi\left(\H_{\C}\right)=\psi\left(\H_{\C^\p}\right)$ and of
the variables associated to the regular edges
that are different from $e$ and $f$. In this case (\ref{eq3})
follows from conditions (\ref{Ieq1}) and (\ref{Ieq2}).

{\bf Case 2.} $(\C,e,f)$ has property (i) but not (ii). This case is
``the dual'' of the previous one and may be handled in a similar
manner. The edge $e$ cannot be a bridge in $\C\cup\H$, otherwise
we may show that $(\C\cup\H\cup\{f\})\setminus
\{e\}$ contains a cycle containing $e$ and $(\C,e,f)$ has property
(ii). Therefore there is at least one
cycle in $\C\cup\H$ containing $e$, and $e$ is internally inactive with
respect to $\C$ under both $\phi$ and $\phi^\p$. Similarly $f$ is
internally inactive with respect to $\C^\p$ under both $\phi$ and
$\phi^\p$. Using $\phi(e)<\phi(f)$ and $\phi^\p(f)<\phi^\p(e)$ we obtain
that $e$ and
$f$ contribute the following factors to the polynomials  $W(G,c,\phi,C)$,
$W(G,c,\phi,C^\p)$, $W(G,c,\phi^\p,C)$, $W(G,c,\phi^\p,C^\p)$:
$$
\begin{array}{c|cc}
&\C&\C^\p\\
\hline
\phi& x_{\lambda} y_{\mu}& x_{\mu}(??)\\
\phi^\p& x_{\lambda} (??)& x_{\mu} y_{\lambda}\\
\end{array}
$$
The double question marks indicate that we are left to determine whether $e$ is
externally active with respect to $\C^\p$ under $\phi$ and whether $f$ is
externally active with respect to $\C$ under $\phi^\p$.  Again the
answer to both questions the answer is simultaneously ``yes'' or ``no'':
to decide we must compare the labels of the remaining edges to the label
of $e$ resp.\ $f$, on the the unique cycle contained in $\C^{\p}\cup\{e\}$
  which is the same as the unique cycle contained in $\C\cup\{f\}$.
If the answer to both remaining questions is ``no'', then the left hand
side of (\ref{eq3}) becomes zero, otherwise (\ref{eq3}) becomes
\begin{equation}
\label{Ecase2}
W_0(x_\lambda y_\mu+x_\mu Y_\lambda -x_\lambda Y_\mu-x_\mu y_\lambda
)\in I,
\end{equation}
which follows from condition (\ref{Ieq1}).

{\bf Case 3.} $(\C,e,f)$ has both properties: (i) and (ii).
Now $\phi(e)<\phi(f)$ implies that $f$ is externally inactive with
respect to $\C$ and internally inactive with respect to
$\C^{\p}$ under $\phi$. Similarly, $\phi^\p(f)<\phi^\p(e)$ implies that
$e$ is internally inactive with respect to $\C$ and externally inactive
with respect to $\C^{\p}$ under $\phi^\p$. We obtain
that $e$ and
$f$ contribute the following factors to the polynomials  $W(G,c,\phi,C)$,
$W(G,c,\phi,C^\p)$, $W(G,c,\phi^\p,C)$, $W(G,c,\phi^\p,C^\p)$:
$$
\begin{array}{c|cc}
&\C&\C^\p\\
\hline
\phi& (?) y_{\mu}& x_{\mu} (??)\\
\phi^\p& x_{\lambda} (??)& (?) y_{\lambda}\\
\end{array}
$$
As in the previous two cases, the question marks indicate missing
activity determination. The argument used in Case 1 is applicable here
to show that $e$ is internally active with respect to $\C$ under $\phi$
if and only if $f$ is internally active with respect to $\C^\p$ under
$\phi^\p$. Hence either both single question marks need to be replaced
by upper case letters, or both need to be replaced by lower case
letters. Similarly, we may reuse the argument from Case 2 to show that
$e$ is externally active with respect to $\C^\p$ under $\phi$ if and
only if $f$ is externally active with respect to $\C$ under $\phi^\p$.
Hence either both double question marks need to be replaced
by upper case letters, or both need to be replaced by lower case
letters. Depending on the choice of letter case for the missing
variables we have $2\times 2=4$ possibilities: either
the left hand side of (\ref{eq3}) becomes zero, or (\ref{eq3})
becomes (\ref{Ecase1}), or (\ref{eq3}) becomes (\ref{Ecase2}), or
(\ref{eq3}) becomes
\begin{equation}
\label{Ecase3}
W_0(X_\lambda y_\mu+x_\mu Y_\lambda -x_\lambda Y_\mu-X_\mu y_\lambda
)\in I.
\end{equation}
We only need to observe that (\ref{Ecase3}) is the same as (\ref{Ieq1}).

We now turn to the necessity of the conditions of the theorem. The
necessity of (\ref{Ieq1}) follows from the double edge graph example
given in \cite{BR} (as shown in Figure \ref{RTuttefig1} earlier but with
both edges as regular edges). For the necessity of (\ref{Ieq2}), let us
look at the graph given in the following figure.
\begin{figure}[!htb]
\begin{center}
\includegraphics[scale=.6]{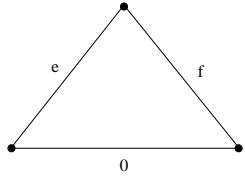}
\caption{\label{TE} A triangle graph $G$ with two regular edges and
one  zero edge.}
\end{center}
\end{figure}
In the figure, let $c(e)=\lambda$, $c(f)=\mu$,
$\phi(e)=1=\phi^\p(f)$, and $\phi(f)=2=\phi^\p(e)$. There are only
three choices for the contracting set $\C$: $\C_1=\{e\}$,
$\C_2=\{f\}$, or $\C_3=\{e,f\}$. With respect to $\C_1$, $e$ is
internally active under $\phi$ because deleting $f$ (since it has a
larger label) will make $e$ a bridge. $f$ is externally inactive
because it does not close a cycle with $\C_1$. Thus their combined
contribution under $\phi$ with respect to $\C_1$ is $X_\lambda
y_\mu$. With respect to $\C_2$, $e$ is externally inactive under
$\phi$ because it does not become a loop after $f$ is contracted
first, and $f$ is internally inactive because it is not a bridge as
no edge can be deleted first. Thus their combined contribution under
$\phi$ with respect to $\C_2$ is $y_\lambda x_\mu$. Finally, with
respect to $\C_3$, $e$ and $f$ are both internally inactive as they
are both non-bridge since there is no any edge to be deleted first.
Thus their combined contribution under $\phi$ with respect to $\C_3$
is $x_\lambda x_\mu$. By definition, we have
\begin{equation}
T_\H(G,\phi)=\sum_{1\le j\le 3} W(G,c,\phi,\C_j)\psi(\H_{\C_j})=
x_\lambda x_\mu \psi(H_0)+(X_\lambda y_\mu+x_\mu y_\lambda)\psi(H_1),
\end{equation}
where $H_0$ is the graph that contains only a zero loop  edge and
$H_1$ is the graph that contains only a simple zero edge. Similarly,
we have
\begin{equation}
T_\H(G,\phi^\p)=\sum_{1\le j\le 3} W(G,c,\phi^\p,\C_j)\psi(\H_{\C_j})=
x_\lambda x_\mu \psi(H_0)+(x_\lambda y_\mu+X_\mu y_\lambda)\psi(H_1).
\end{equation}
So in order to have $T_\H(G,\phi)=T_\H(G,\phi^\p)$ (with $\psi=1$),
we must have $X_\lambda y_\mu+x_\mu y_\lambda-x_\lambda y_\mu-X_\mu
y_\lambda\in I$. But this, together with (\ref{Ieq1}), implies
(\ref{Ieq2}). This finishes our proof of the theorem.
\end{proof}

From now on, we will assume $T_\H(G,\phi)$ is defined in
$\R[\Lambda]/I_1$ so that it is independent of the proper
labelings of $G$ and we will simply write $T_H(G)$ for
$T_H(G,\phi)$.  An immediate consequence of this fact is the
following corollary.

\begin{corollary}
$T_\H(G)$ can be computed via the following recursive formula:
\begin{equation}\label{recur1}
T_\H(G)=\left\{
\begin{array}{ll}
y_\lambda T_\H(G\setminus e) + x_\lambda T_\H(G/e),& \ {\rm
if}\ e\ {\rm is\  not\ a\ bridge\ nor\ a \ loop,}\\
X_\lambda T_\H(G/e), & {\rm
if}\ e\ {\rm is\  a\ bridge,}\\
Y_\lambda T_\H(G\setminus e), & {\rm if}\ e\ {\rm is\  a\
loop.}
\end{array}
\right.
\end{equation}
In the above, $e\not\in\H$ is a regular edge, $\lambda=c(e)$,
$G\setminus e$ is the graph obtained from $G$ by deleting $e$
and $G/e$ is the graph obtained from $G$ by contracting $e$.
\end{corollary}

\begin{proof}
We may assign $e$ the largest label. So in the case that $e$ is
not a bridge nor a loop, it is always inactive. Thus we have
\begin{eqnarray*}
T_\H(G)&=& \sum_{\C} W(G,c,\phi,\C)\psi(\H_\C)\\
&=& \sum_{e\in \C} W(G,c,\phi,\C)\psi(\H_\C)+\sum_{e\not\in \C}
W(G,c,\phi,\C)\psi(\H_\C)\\
&=& x_\lambda\sum_{\C^\p}
W(G/e,c,\phi,\C^\p)\psi(H_{\C/e})+y_\lambda\sum_{\C}
W(G\setminus e,c,\phi,\C\setminus e)\psi(\H_{\C})\\
&=& x_\lambda T_\H(G/e) + y_\lambda T_\H(G\setminus e),
\end{eqnarray*}
If $e$ is a bridge, then it has
to be in $\C$ (the cocycle $\{e\}$ cannot by contained in $\D$ by
Definition~\ref{TD1}) and it must be internally active by
Remark \ref{equi-act-def}. So its contribution for any chosen $\C$
is always $X_\lambda$. Finally, if $e$ is a loop, then it has to
be in $\D$ (the cycle $\{e\}$ cannot be contained in $\C$ by
Definition~\ref{TD1}) and is always externally active
by Remark \ref{equi-act-def}, so its contribution for any $\C$ will
be $Y_\lambda$.
\end{proof}

\begin{remark}\label{remark3.20}{\em
A careful reader may have realized that the ideal $I_1$ in Theorem
\ref{DH1} has to be replaced by the ideal $I_0$ generated by the
differences listed in Proposition \ref{T_BR} if we restrict ourselves to
the class of edge sets $\H$ containing only bridges and loops. This
is because in this case the activities of the regular edges will not be
affected by the zero edges at all (and the graphs $\H_\C$ are all
isomorphic).
}
\end{remark}

We conclude this section with an interesting observation that was
communicated to us by Sergei Chmutov~\cite{Ch-p}. In the ring
$\R[\Lambda]/I_1$, the relations (\ref{Ieq1}) and (\ref{Ieq2})
become equations and may be used to eliminate all
variables $X_{\lambda}$ and $Y_{\lambda}$. Indeed, first we should note
that all relations (\ref{Ieq1}) and (\ref{Ieq2}) generate the same ideal
as all relations (\ref{Ieq2}) and all relations of the form
\begin{equation}\label{Ieq1-bis}
\det\left(\begin{array}{ll}
X_\lambda & y_\lambda\\
X_\mu & y_\mu
\end{array}\right)
- \det\left(\begin{array}{ll}
x_\lambda & y_\lambda\\
x_\mu & y_\mu
\end{array}\right)\in I.
\end{equation}
Since  $I_1$ is a prime ideal (by Lemma~\ref{L_ip}), the ring $\R[\Lambda]/I_1$
is a domain which we may localize by the semigroup generated by all
variables $x_{\lambda}, y_{\lambda}$ ($\lambda\in\Lambda$).
In this localized ring, the relations
(\ref{Ieq1-bis}) may be rewritten as
\begin{equation}\label{CIeq1-bis}
\frac{X_\lambda-x_\lambda}{y_\lambda}
=
\frac{X_\mu-x_\mu}{y_\mu}\quad\mbox{for all $\lambda,\mu$,}
\end{equation}
whereas the relations (\ref{Ieq2}) may be rewritten as
\begin{equation}\label{CIeq2}
\frac{Y_\lambda-y_\lambda}{x_\lambda}
=
\frac{Y_\mu-y_\mu}{x_\mu}\quad\mbox{for all $\lambda,\mu$.}
\end{equation}
Introducing $X$ for the common value of all
$(X_\lambda-x_\lambda)/y_\lambda$ and $Y$ for the common value of all
$(Y_\lambda-y_\lambda)/x_\lambda$, we obtain the following.
\begin{theorem}
The ring $\R[\Lambda]/I_1$, localized by all variables $x_{\lambda},
y_{\lambda}$ ($\lambda\in\Lambda$),
is isomorphic to the polynomial ring $R[
x_{\lambda},y_{\lambda}\::\: \lambda\in \Lambda][X,Y]$, localized by all
variables $x_{\lambda}, y_{\lambda}$ ($\lambda\in\Lambda$).
Under this isomorphism,
each $X_{\lambda}$ corresponds to $x_{\lambda}+Xy_{\lambda}$ and
each $Y_{\lambda}$ corresponds to $y_{\lambda}+Yx_{\lambda}$.
\end{theorem}
Therefore the colored relative Tutte may be considered as an element of
a localized polynomial ring that does not need to be factorized by any
algebraic relation.
\begin{remark}{\em
For some other (non-relative) generalizations of the Bollob\'as-Riordan
colored Tutte polynomial of signed graphs, the idea of eliminating all
variables $X_{\lambda}$ and $Y_{\lambda}$ appears in the work of
Ellis-Monaghan and Traldi, see \cite[Corollary 5.2]{ET} and
\cite[Corollary 1.3]{Tr}. In the situations considered
by these sources the Tutte polynomial originally considered is not
an element of (a variant of) $\R[\Lambda]/I_1$, but of (a variant of)
$\R[\Lambda]/I_0$. Thus assuming that one may take the inverses of the
variables $x_{\lambda}$ and $y_{\lambda}$ in these situations involves
``giving up a certain degree of generality''. The same applies to the
relative Tutte polynomial in the situation mentioned in
Remark~\ref{remark3.20}.}
\end{remark}

%\bigskip
\section{Disconnected graphs and some applications}\label{s4}

\subsection{Disconnected graphs}
Our first example below shows that the colored Tutte polynomial
of a disconnected graph introduced by Bollob\'as and Riordan~\cite[Section 3.4]{BR} is
equivalent to the relative Tutte polynomial of a related
{\em connected} graph with a suitably chosen function $\psi$.

\begin{example}{\em
Let $\alpha_1,\alpha_2,\ldots$ be an infinite list of variables, let
$\R={\mathbb Z}[\alpha_1,\alpha_2,\ldots]$ and
let $\psi$ be the mapping that associates $\alpha_{k+1}$ to each graph that has $k$ edges. Consider a disconnected graph $G$ whose
edges are colored by a color set $\Lambda$. Select a vertex in each component and add a tree on the selected vertices and color each edge on this added tree with a color $0\not\in\Lambda$. Call the resulting connected graph $\widetilde{G}$. Let $\H$ be the set of the zero
colored edges. By Remark \ref{remark3.20}, the relative Tutte polynomial $T_\H(\widetilde{G})$
is exactly the disconnected Tutte polynomial
introduced by Bollob\'as and Riordan~\cite[Section 3.4]{BR} (as an
element of ${\mathbb Z}[\alpha_1,\alpha_2,\ldots][\Lambda]/I_0$).
}
\end{example}

Although the above example indicates a subtle trick that allows us to
reduce the study of disconnected graphs to connected ones, it is more
straightforward to generalize our notion of a relative Tutte polynomial
to disconnected graphs as follows. Let $G$ be any graph, and $\H$ a
subset of $E(G)$. We define the deleting and contracting sets as in
Definition~\ref{TD1}, keeping in mind that a cocycle in a disconnected
graph is a minimal set of edges whose removal increases the number of
connected components. We introduce colors and a labeling on the regular
edges as before. Regarding the map $\psi$, we amend
Definition~\ref{Dpsi} to require that the value of
$\psi$ should remain the same if we rearrange the blocks of a graph {\em
  within the same connected component}. To make this notion precise,
recall that {\em vertex splicing} is an operation that merges two
disjoint graphs by picking a vertex from each and identifying these
selected vertices, thus creating a cut point.
The opposite operation is {\em vertex splitting} that creates two disjoint
graphs by replacing a cut point $v$ with two copies $v_1$ and $v_2$, and
makes each block containing $v$ contain exactly one of $v_1$ and $v_2$.
\begin{definition}
{\em Let $G$ be a graph that has a cut point $u$. A {\em vertex pivot} is
  a sequence of vertex splitting and vertex splicing as follows. First
  we split $G$ by creating two copies of $u$ and two disjoint graphs
  $G_1$ and $G_2$. Then we take a vertex $v_1\in V(G_1)$ from the
  connected component of $u_1$ and a vertex $v_2\in V(G_2)$ in the
  connected component of $u_2$ and we merge $G_1$ and $G_2$ by
  identifying $u_1$ with $u_2$.
}
\end{definition}
We require that $\psi$ be a mapping that assigns to each graph an
element in an integral domain $\R$ and that the value of $\psi$ remains
unchanged if we perform a vertex pivot on its input.
For example, assigning to each
graph the number of its vertices is such an operation.
We may now adapt Lemma~\ref{HC} by observing that after dropping the
requirement of $G$ being connected, the graph  $\H_{\C^\p}$ is either
equal to $\H_{\C}$ or may be obtained from it by a single vertex pivot.
The definition of
the relative Tutte polynomial is then given by (\ref{eq1}) and
Theorem~\ref{DH1} holds for this Tutte polynomial as before.

\begin{example}{\em
Let $\alpha_1,\alpha_2,\ldots$ be an infinite list of variables, let
$\R={\mathbb Z}[\alpha_1,\alpha_2,\ldots]$ and
let $\psi$ be the mapping that associates $\alpha_k$ to each graph that
has $k$ connected components (so $\psi$ is invariant under the vertex pivot operation). Consider a disconnected graph $G$ with colored edges.  Let $G_1$, $G_2$, ..., $G_m$ be the connected components of $G$ and assume
that $\H_j$ is a subgraph of $G_j$ and let $\H=\cup_{1\le j\le
m}\H_j$. The relative Tutte polynomial $T_\H(G)$
is then
$\alpha_{m}\prod_{1\le j\le m} T_{\H_j}(G_j)$, since the
deletition-contraction process does not change the number of components
in $G$. If $\H$ contains only bridges and loops, then $I_1$ is replaced
with $I_0$. In particular, if $\H=\emptyset$, then we obtain exactly
$\alpha_{k(G)}\cdot T(G)$.
}
\end{example}

\begin{example}
{\em
Let $G$ be a graph whose edges are all regular of the same color $\lambda$.
Add a loop of color zero to each vertex of $G$ to get the graph
$\widetilde{G}$. Let $\psi$ be the mapping that associates $(-1)^e\cdot
(-x)^k$ to each graph having $e$ (zero-colored) edges and $k$ connected
components. Thus $\psi$ maps into the polynomial ring ${\mathbb
  Z}[x]$. Substituiting $1-x$ into $X_{\lambda}$ and $x_{\lambda}$
and $0$ into $Y_{\lambda}$ and $y_{\lambda}$ yields
$$
T_{\H}(\widetilde{G})\left|_{\begin{array}{l} \scriptstyle
    X_{\lambda}=x_\lambda=1-x\\
\scriptstyle    Y_{\lambda}=y_\lambda=0\end{array}}\right. =
(-1)^{v(G)}\cdot (-x)^{k(G)}T(G)(1-x,0)=(-1)^{v(G)-k(G)}\cdot
x^{k(G)}T(G)(1-x,0).
$$
Here $T(G)(x,y)$ is the ordinary Tutte polynomial of $G$, and so
$T_{\H}(\widetilde{G})$ generalizes the chromatic polynomial of $G$ (see
\cite[Cha. X, Section 4, Theorem 6]{B}).
}\end{example}

\begin{example}{\em
Let $G$ be a graph with a special edge set $\H$. Let $\Lambda_1$ and $\Lambda_2$ be two color sets. Assume that the edges of $E(G)\B \H$ are colored with colors from the set $\Lambda_1$ and that the edges of $\H$ are colored with colors from the set $\Lambda_2$. Define $\psi(\H_\C)$ to be the ordinary Tutte polynomial of the graph $\H_\C$ (as element of $\Z[\Lambda_2]/I_2$ where $I_2$ is the ideal generated by polynomials of the form (\ref{E_I11}) and (\ref{E_I12}) using colors from $\Lambda_2$), then the ordinary Tutte polynomial $T(G)$ (as an element of $\Z[\Lambda_1,\Lambda_2]/I$ where $I$ is the ideal generated by polynomials of the form (\ref{E_I11}) and (\ref{E_I12}) using colors from $\Lambda_1\cup\Lambda_2$) equals the relative Tutte polynomial $T_\H(G)$. }
\end{example}
We leave the verification of the above example to our reader.

\subsection{The set-pointed Tutte polynomial of Las Vergnas}

Given a matroid $M$ on the set $E$, {\em pointed} by a subset
$A\subseteq E$, Las Vergnas~\cite[Eq. (3.1)]{Ver} defines the {\em
 Tutte polynomial of $M$ pointed by $A$} as the $3$-variable polynomial
\begin{equation}
t(M;A; x,y,z)=
\sum_{X\subseteq E\setminus A}
(x-1)^{r(M)-r_M(X\cup A)}
(y-1)^{|X|-r_M(X)}z^{r_M(X\cup A)-r_M(X)}.
\end{equation}
Here $r(M)$ is the rank of $M$ and $r_M$ is the rank function.
As noted in \cite[p. 978]{Ver}, the choice $A=\emptyset$ yields the
the ordinary Tutte polynomial, whereas the choice $|A|=1$ yields
the Tutte polynomial of $M$ pointed by a single edge $e$, introduced
by Brylawski~\cite{Bry0,Bry}.

\begin{lemma}
\label{L_Ver}
Let $G$ be a graph with edge set $E$, pointed by the subset $A$ of
$E$. Color all elements of $E\setminus A$ with the same color $\lambda$
and all elements of $A$ with the color zero, i.e., let $\H$ be the
subgraph whose set of edges is $\H$. Let us also define
$\psi(G)=z^{r(G)}$ where $r(G)$ is the rank of $G$.
Then the Tutte polynomial
$t(M;A; x,y,z)$ of $G$ pointed by $A$ may be obtained from the relative
Tutte polynomial $T_\H(G)$ by substituting $x_{\lambda}\mapsto 1$,
$y_{\lambda}\mapsto 1$, $X_{\lambda}\mapsto x$ and $Y_{\lambda}\mapsto
y$.
\end{lemma}
Indeed, both $T_\H(G)$ and $t(M;A; x,y,z)$ satisfy analogous
deletion-contraction rules for edges $e\in E\setminus A$, and the
statement is trivially true when $A=E$. Lemma~\ref{L_Ver} may be
generalized from graphs to all (set-pointed) matroids without any
substantial change.

As noted in \cite[p. 987]{Ver}, an ``almost equivalent'' problem to
computing the Tutte polynomial of a set-pointed matroid is the problem
of finding the Tutte polynomial~\cite[(5.1)]{Ver} of a {\em matroid
  perspective} (for a definition, see~\cite[p. 977]{Ver}). Applications
of the matroid perspective approach may be found
in~\cite{Ver00,Ver01,Ver02}. As a consequence of Lemma~\ref{L_Ver}, a
colored generalization of these applications could be considered,
using the relative Tutte polynomial.

Another generalization of the Tutte polynomial of a set-pointed matroid
is the Tutte polynomial of a matroid perspective sequence, studied by
Chaiken~\cite{Cha}, the term ``matroid perspective sequence'' is used by
Las Vergnas~\cite[p. 975]{Ver}. Finding a common generalization of the
relative Tutte polynomial of a colored matroid and of the Tutte
polynomial of a matroid perspective sequence seems an interesting
question for future research.

\subsection{The random-cluster model}

We conclude this section with an application of the relative Tutte polynomial to the random-cluster model considered in \cite{FK}.
A random-cluster model can be thought of as a graph $G(V,E)$ that is associated with a function $p:\ E\longrightarrow [0,1]$. We may think of $p(e)$ as the probability that the edge $e\in E$ ``survives an accident", and $q(e)=1-p(e)$ for the probability that the edge $e$
``breaks'' in an accident. Fortuin and Kasteleyn~\cite{FK} introduced
the following polynomial of the variable $\kappa$
as a cluster generating function $Z(G; p, \kappa)$:
\begin{equation}
\label{E_z}
Z(G; p, \kappa) = \sum_{C \subseteq E} p^C q^{E \setminus C}
\kappa^{k(C)}.
\end{equation}
Here $p^C$ is a shorthand for the product $\prod_{e\in C} p_e$,
$q^{E \setminus C}$ is a shorthand for the product $\prod_{e\in
  E\setminus C} q_e$, and $k(C)$ is the number of connected components
in the subset of edges $C$. The quantity
$p^C q^{E \setminus C}$ represents the probability
that exactly the subset of connections $C$ survives. Thus the function
$Z(G; p, \kappa)$ gives the expectation $E(\kappa^{k(G)})$, which is an indirect measure of the mean average number of connected components of the network after an accident happens. The following example concerns only a very special network setting since we do not intend to explore the general cases in this paper.

\begin{example}{\em
Let us consider the following problem. Suppose that a network of communications is to be built among a number of ``stations" and two types of communication methods are available. Type A is cheap and fast (such as the internet) so communication between two stations does not need to be directly so long as it can be routed through a sequence of stations connected by type A communication. However, type A communication may break in the event of some ``accident". On the other hand, type B communication is dependable and will not break in the event of an accident. However, it is costly, slow and hence communication between two stations using type B cannot be routed through a sequence of stations where at least two pairs of stations are communicated by type B method. If station 1 and station 2 cannot be communicated through a sequence of stations connected by type A communication, then either a direct type B communication has to be established between them, or they have to be routed to two stations 3 and 4 through type A communications first and then through a direct type B communication between station 3 and station 4, if the latter is cheaper. It thus makes sense to build a network with a two layer structure: a primary structure that uses type A communication to connect all stations and a secondary structure that would enable a type B communication between any two stations when needed after an accident strikes.

\medskip
This two layer structure network can then be represented by a graph $G$: the stations are the vertices of $G$ and a type A communication between two stations is a regular edge and a type B communication between two stations is a ``zero edge". Each zero edge $f$ comes with a positive weight $c_f$, namely the cost to operate that communication line (and there is a zero edge between any two vertices so the zero edges alone give us a complete graph), and each regular edge $e$ comes with a positive weight $p_e<1$, namely the probability that edge $e$ survives in the event of an accident (so $q_e=1-p_e$ is the probability that the edge breaks in the event of an accident). We will assume further that the events the regular edges break are all independent of each other. We wish to compute the mean cost of maintaining a functional network in the event of an accident. And this turns out to be a relative Tutte polynomial with a suitable choice of $\psi$ and the assignment of the variables to the regular edges.

\medskip
Since the zero edges alone give us a complete graph, no regular edges will be active. For a regular edge $e$, we will then assign $x_e=p_e$ and $y_e=q_e$ (we are using $e$ as the color of $e$ since every regular edge is treated as being colored differently). Notice that for a given contracting set $\C$, there is at least one edge between any two vertices in the graph $\H_\C$ (though $\H_\C$ may now have loop edges and multiple edges). If we remove all loop edges from $\H_\C$, and for any two vertices of $\H_\C$, keep only the edge incident to them that has the smallest weight among all edges incident to these two vertices, then we obtain a complete graph with the minimum total weight (the summation of all weights in the graph), which we will denote as $\H_\C^\p$. We then define $\psi(\H_\C)$ to be the total weight of $\H_\C^\p$. Of course $\psi$ so defined is invariant under the vertex pivot operation. $T_\H(G)$ then represents precisely the mean cost of operating a functional network in the event of an accident.
}
\end{example}

\section{Applications to Virtual Knot Theory}\label{s5}

In this section we apply the relative Tutte polynomial to
virtual knot theory. More specifically, we relate
the relative Tutte polynomial to the Kauffman bracket polynomial
of a virtual knot. It is well known that a classical link diagram
can be converted to a plane graph (called the {\em face graph} or the {\em Tait graph} of the diagram) where the edges of the face
graph are colored with the color set $\{+,-\}$. Furthermore, the
Tutte polynomial of this colored plane graph can be converted to
the Kauffman bracket polynomial via suitable variable
substitutions. However, for a virtual link diagram, the
corresponding face graph created using the old approach creates
some special ``zero" edges which cannot be handled via the
traditional deletion-contraction approach when one tries to define
or compute the Tutte polynomial of such a graph.
One way to overcome this difficulty is to change the virtual link diagram
into a ribbon graph where there are no more zero edges, and the
Bollob\'as-Riordan polynomial~\cite{BR2,BR3} of the ribbon graph may be
used to express the Jones polynomial \cite{Ch,CP,CV,Ka1,Ka2}.
The earlier results along these lines \cite{CP,Ka1,Ka2} are only applicable to
``checkerboard colorable" virtual link diagrams. Figure~\ref{nonchecker}
shows the virtual trefoil, which is not checkerboard colorable. The best
current generalization in this direction is due to Chmutov and Voltz
\cite{Ch,CV}, providing a formula for {\em all} virtual links.
Our approach is closer to Kauffman's~\cite{K2}, and uses
only the underlying graph structure of the face graph.

First, we need some preparation on how we will handle a graph with
only the zero edges. In other words, we would like to choose the
function $\psi$ in a way so that we may apply the relative Tutte
polynomial to a virtual link diagram.

\subsection{The face graph of a link diagram}

A regular link diagram $K$ can be viewed as a plane 4-valent graph and
from which one can obtain a so called ``face graph" $G$ of it. Here is a brief description of this process. One starts from the regular
projection $K$ and shade the regions in its projection either
``white'' or ``dark'' in a checkerboard fashion, so that no two
dark regions are adjacent, and no two white regions are adjacent (this can always be done for a 4-valent plane graph).
We usually consider the infinite region surrounding the knot
projection to be white. Note that as we move diagonally over a
knot crossing, we go from a white region to a white region, or
from a dark region to a dark region.  Next we construct a dual
graph of $K$ by converting the dark regions in $K$ into vertices
in a graph $G$ and converting the crossings in $K$ between two
dark regions into edges incident to the corresponding vertices in
$G$. So if we can move diagonally over a knot crossing from one
dark region to another, then these two dark regions and the
crossing will be represented in $G$ as two vertices connected by
an edge. Note that we may obtain parallel edges from some knot
projections. Now we have our unsigned graph. To obtain the signed
version, we look at each crossing in the knot projection.  If,
after the upper strand passes over the lower, the dark region is
to the left of the upper strand, then we denote this as a positive
crossing.  If the dark region is to the right of the upper strand,
we denote it as a negative crossing. See Figure \ref{crossingsign}.
\begin{figure}[!htb]
\begin{center}
\includegraphics[scale=0.5]{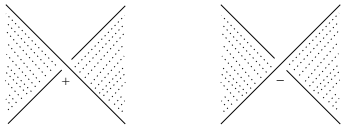}
\caption{The sign of a classic crossing with respect to a checkerboard coloring in a link diagram.}\label{crossingsign}
\end{center}
\end{figure}
Then our signed graph is
obtained by marking each edge of $G$ with the same sign as the
crossing of $K$ to which it corresponds. In the case of a virtual crossing, we cannot assign the $\pm$ to it so we will simply assign it the number $0$.
Figure \ref{nonchecker} shows such an example.
\begin{figure}[!htb]
\begin{center}
\includegraphics[scale=0.7]{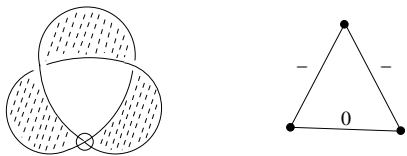}
\caption{The virtual trefoil knot (which is not checkerboard colorable) and its face graph.}\label{nonchecker}
\end{center}
\end{figure}

\subsection{Zero order of a plane graph}
As we have seen from the above discussion, a face graph of a virtual
link diagram will contain two kinds of edges: the edges that correspond
to classic crossings (the regular edges) and the edges that correspond
to virtual crossings (the zero edges). Since we cannot carry out the
typical crossing splitting operations at a virtual crossing (which are
used in defining all the knot polynomials), it means that we cannot
perform the typical contraction/deletion operation on the zero edges
(which is our motivation of introducing the relative Tutte
polynomials). In order to define an appropriate relative Tutte
polynomial (that can lead us to the Kauffman bracket polynomial) for a
face graph of a virtual link diagram, we have to choose a proper $\psi$
defined on graphs with only the zero edges (such graphs are face graphs
of virtual link diagrams with only virtual crossings). Let $G^\p$ be
such a plane graph (it is not necessarily connected). Since it is the
face graph of a (virtual) link diagram, its number of components (i.e.,
the number of components of the link) is a well defined number. We call
this number the {\em zero order} of the graph $G^\p$ and denote it by
$|G^\p|_0$. By a result due to Las Vergnas \cite{Ver2},
$|G^\p|_0=\log_2|T_{G^\p}(-1,-1)|+1$, where $T_{G^\p}(x,y)$ is the
ordinary (non-colored) Tutte polynomial of $G^\p$. $|G^\p|_0$ can also
be determined using the following simplification operations called {\em
  zero edge operations} or simply $0$-operations.

\begin{figure}[!htb]
\begin{center}
\includegraphics[scale=.6]{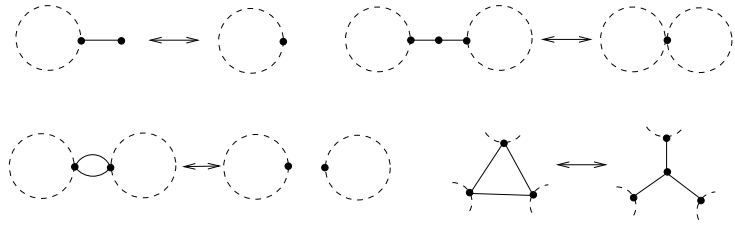}
\caption{\label{RTuttefig2} The $0$-operations.}
\end{center}
\end{figure}

\begin{lemma}
For any given plane graph $G$, there exists a finite sequence $S$
of $0$-operations that leads to a graph $G_S$ with only vertices.
Furthermore, the number of vertices in $G_S$ is equal to $|G^\p|_0$.
\end{lemma}

\begin{proof}
The $0$-operations, translated into the link diagrams, are simply the Reidemeister moves. These moves never change the number of components in the link diagram and there exists
a sequence of Reidemeister moves that will take the original
diagram to disjoint circles (since there are no restrictions on the moves in this case), which is equivalent to a plane graph with only vertices.
\end{proof}

We now choose $\psi(G)=d^{n-1}$ where $n=|G|_0$.
Notice that $\psi$ so defined is invariant under vertex pivot, since the
vertex pivot in $G$, when translated in terms of the knot diagram, is
equivalent to taking connected sum of two link diagrams at two different
places, which of course does not affect the number of components in the
connected sum (which is always the sum of the numbers of components in
each link diagram minus one). Under this choice of $\psi$, the Tutte
polynomial for a disconnected (face) graph will then have the form
stated in the following proposition.

\begin{proposition}\label{pro5.2}
Let $G$ be a graph with connected components $G_1$, $G_2$, ..., $G_m$ and assume that $\psi$ is as defined above. Assume
that $\H_j$ is a subgraph of $G_j$ and let $\H=\cup_{1\le j\le
m}\H_j$, then $T_\H(G)$ is
\begin{equation}\label{e51}
T_\H(G)=d^{m-1}\prod_{1\le j\le m}T_{\H_j}(G_j),
\end{equation}
where $d$ is the same variable as in the definition of $\psi$.
\end{proposition}

\begin{proof}
Each contracting set $\C$ of $G$ can be uniquely written as $\C=\cup_{1\le j\le m}\C_j$ where $\C_j$ is a contracting set in $G_j$. Assume that $|\H_{\C_j}|_0=k_j$, then $|\H_{\C}|_0=\sum_{1\le j\le m}k_j$. Thus the total contribution of $\H_\C$ to $T_\H(G)$ is $d^{-1+\sum_{1\le j\le m}k_j}$. On the other hand, each ${\H_j}_{\C_j}$ contributes a factor to $d^{k_j-1}$ to $T_{\H_j}(G_j)$. The proposition statement now follows from the equality $d^{-1+\sum_{1\le j\le m}k_j}=d^{m-1}d^{k_1-1}d^{k_2-1}\cdots d^{k_m-1}$.
\end{proof}

\begin{remark}
{\em It is worth noting that the zero order is defined for {\em plane
    graphs only}, thus the relative Tutte polynomial introduced in this
    section is also {\em defined on the class of plane graphs
    only}. This should not represent a problem, since the class of plane
    graphs is closed not only under deletion and contraction, frequently
    used in all Tutte polynomial calculations, but also under the vertex pivot
operation, introduced in connection with the map $\psi$. A
    generalization of the Bollob\'as-Riordan polynomial to minor-closed
    classes of matroids was developed by Ellis-Monaghan and
    Traldi~\cite{ET}. One would need to start with the class of plane
    graphs and such a generalized Bollob\'as-Riordan polynomial and then
    adapt the reasoning of the preceding sections.}
\end{remark}

\subsection{Converting the relative Tutte polynomial to the
Kauffman bracket polynomial}

The following theorem is the main motivation of this paper.

\begin{theorem}
Let $K$ be a virtual link diagram and let $G$ be its face graph obtained from $K$ where a virtual crossing in $K$ corresponds to a zero edge in $G$. Let $\H$ be the
subgraph of $G$ that contains all the zero edges, then $T_\H(G)$, as defined in (\ref{e51}),
equals the Kauffman bracket polynomial through the
following variable substitution:
\begin{eqnarray*}
&&X_+\rightarrow -A^{-3},\ X_-\rightarrow -A^3,\ Y_+\rightarrow
-A^3,\ Y_-\rightarrow -A^{-3}\\
&& x_+\rightarrow A,\ x_-\rightarrow A^{-1},\ y_+\rightarrow
A^{-1},\ y_-\rightarrow A,\\
&& d\rightarrow -(A^2+A^{-2}).
\end{eqnarray*}
\end{theorem}

\begin{proof}
Let $K$ be any virtual link diagram with $n$ regular crossings and
let $G$ be its face graph. We will prove the theorem by induction on $n$.

For $n=0$, $\langle K\rangle=d^{m-1}$ by definition where $d=-(A^2+A^{-2})$ and $m$ is the number of components in the link diagram $K$. In
this case, $G$ is a plane graph with only zero edges and
$T_\H(G)=\psi(\H)=d^{m-1}$ as well by the definition of $\psi$. So we have $\langle K\rangle=T_\H(G)$.

Assume now that we have $\langle K\rangle=T_\H(G)$ for any $K$ with $n\ge 0$
regular crossings. We would like to show that for any $K$ with
$n+1$ regular crossings, this is still the case.

So let $K$ be any given virtual link diagram with $n+1$ regular
crossings and let $G$ be its face graph. Since $n+1\ge
1$, there exists at least one regular edge in $G$. Let $e$ be a
regular edge. In the following proofs, we will assume that $e$ is
positive. The case of $e$ being negative can be proved similarly
and is left to the reader. Again $\H$ is the set of all the zero edges of $G$.

Case 1. $e$ is not a bridge nor a loop in $G$. By the recursive
formula (\ref{recur1}), we have
$$
T_\H(G)=x_+T_\H(G/e)+y_+T_\H(G\setminus
 e)=AT_\H(G/e)+A^{-1}T_\H(G\setminus e).
$$
On the other hand, a similar recursive formula of the Kauffman
bracket about the crossing in $K$ that is corresponding to $e$
gives (see \cite{K3} for details about the properties of the bracket polynomial)
$$
\langle K\rangle=\langle \includegraphics[scale=.6]{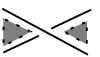}\rangle=A\langle \includegraphics
[scale=.6]{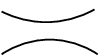}\rangle
+A^{-1}\langle \includegraphics[scale=.6]{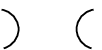}\rangle.
$$
But the face graph of the link diagram corresponding
to the $\includegraphics[scale=.6]{Asplit}$ split is $G/e$
and the face graph of the link diagram corresponding
to the $\includegraphics[scale=.6]{Bsplit}$ split is
$G\setminus e$, thus by our induction hypothesis, we have
$$
T_\H(G)=AT_\H(G/e)+A^{-1}T_\H(G\setminus
e)=A\langle \includegraphics[scale=.6]{Asplit}\rangle
+A^{-1}\langle \includegraphics[scale=.6]{Bsplit}\rangle
=\langle \includegraphics[scale=.6]{cross1}\rangle=\langle K\rangle.
$$

Case 2. $e$ is a loop. By (\ref{recur1}), we have
$$
T_\H(G)=Y_+T_\H(G\setminus e)=-A^3 T_\H(G\setminus e).
$$
On the other hand, we have
$$
\langle K\rangle=\langle \includegraphics[scale=.6]{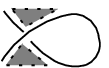}\rangle=A
\langle \includegraphics[scale=.6]{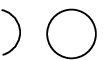}\rangle+
A^{-1}\langle\ \includegraphics[scale=.6]{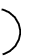}\ \rangle=
(Ad+A^{-1})\langle\ \includegraphics[scale=.6]{Lsplit2}\ \rangle=-A^3
\langle\ \includegraphics[scale=.6]{Lsplit2}\ \rangle.
$$
But the diagram $\ \includegraphics[scale=.6]{Lsplit2}\ $
obtained from $K$ (by deleting the loop) has face
graph $G\setminus e$ and the result follows from the induction
hypothesis again.

Case 3. $e$ is a bridge. By (\ref{recur1}) and our induction
hypothesis, we have
$$
T_\H(G)=X_+T_\H(G/e)=-A^{-3} T_\H(G/e)=-A^{-3}\langle K_0\rangle
$$
where $K_0$ is the link diagram obtained from $K$ by the
$\includegraphics[scale=.6]{Asplit}$ split. On the other hand,
we have
\begin{equation}
\langle K\rangle=\langle\includegraphics[scale=.6]{cross1}\rangle
=A\langle\includegraphics[scale=.6]{Asplit}\rangle
+A^{-1}\langle\includegraphics[scale=.6]{Bsplit}\rangle\label{br}
\end{equation}
Since $e$ is a bridge, the
$\includegraphics[scale=.6]{Bsplit}$ split above creates two
disjoint link diagrams
$K_1=\includegraphics[scale=.6]{Lsplit2}$ and
$K_2=\includegraphics[scale=.6]{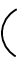}$. Furthermore,
$K_0=K_1\# K_2$. It is well known that $\langle K_1\# K_2\rangle=\langle K_1\rangle\cdot
\langle K_2\rangle$. Thus it follows that
$$
\langle\includegraphics[scale=.6]{Bsplit}\rangle=\langle K_1\sqcup
K_2\rangle=d\cdot\langle K_1\rangle\cdot \langle K_2\rangle=d\cdot\langle K_1\#
K_2\rangle=d\cdot\langle\includegraphics[scale=.6]{Asplit}\rangle.
$$
Combining this with (\ref{br}), we have
$$
\langle K\rangle=(A+dA^{-1})\langle\includegraphics[scale=.6]{Asplit}\rangle=
-A^{-3}\langle\includegraphics[scale=.6]{Asplit}\rangle=T_\H(G).
$$
This finishes our proof.
\end{proof}

\begin{example}{\em
The following is a simple virtual knot diagram $K$ with two virtual crossings marked (which are circled in the diagram), together with its face graph. The edges marked with $0$ correspond to the virtual crossings. At the far right are the remaining graphs at the end of contraction/deletion process of the regular edges. It is easy to see that $\psi(G^\p)=d$ and $\psi(G^{\p\p})=1$. Thus we have
\begin{eqnarray*}
T_\H(G)&=&y_+^2(X_++x_+)\psi(G^\p)+(x_+y_+X_++x_+^2y_++x_+^2Y_+)\psi(G^{\p\p})\\
&=& y_+^2(X_++x_+)d+(x_+y_+X_++x_+^2y_++x_+^2Y_+)\\
&=& (y_+^2d+x_+y_+)(X_++x_+)+x_+^2Y_+\\
&=& (-(A^2+A^{-2})A^{-2}+AA^{-1})(-A^{-3}+A)-A^5\\
&=& -A^{-3}+A^{-7}-A^5=\langle K\rangle.
\end{eqnarray*}
Since the writhe of the diagram is 3, it follows that the Jones polynomial of $K$ is $J_K(t)=(-A^{-3})^3(-A^{-3}+A^{-7}-A^5)|_{A=t^{-1/4}}=t+t^3-t^4$.
\begin{figure}[!htb]
\begin{center}
\includegraphics[scale=.6]{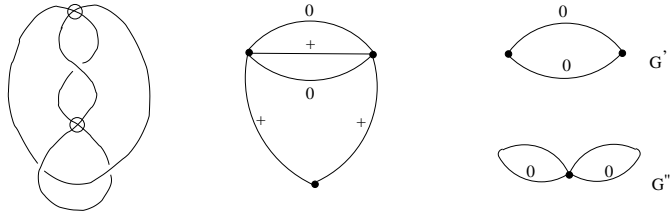}
\caption{\label{virtual2} A virtual knot diagram with two virtual crossings and its face graph.}
\end{center}
\end{figure}
}
\end{example}

\begin{remark}{\em
 The introduction and development of the relative Tutte polynomial and its connection to the Kauffman bracket (and hence Jones) polynomial of virtual links will make the generalization of some existing results in classical knot theory to virtual knot theory possible. For instance, the results of the authors on the colored Tutte polynomials of colored graphs through repeated tensor product operation will generalize to the relative Tutte polynomials without much difficulty \cite{DGH2}. Consequently, the Jones polynomials of virtual knots and links obtained through repeated tangle replacement operation (as discussed in \cite{DEZ}) can be computed in polynomial time. The authors intend to further explore these and other applications of the relative Tutte polynomials in the near future.
}
\end{remark}

\smallskip
\section*{Acknowledgement}
This work was partially supported by NSF grant DMS-0712958 to Y.~Diao
and by NSA grant H98230-07-10073 to G.~Hetyei. The authors wish to thank
Professor Louis Kauffman for introducing the virtual knot theory to
them, Professor Douglas West for providing some useful information on
graph operations, Professors Xian'an Jin and Fuji Zhang for bringing
the results of \cite{Ver2} to their attention, and Professor Sergei
Chmutov for advice and many useful comments. Finally, the authors wish
to express their gratitude to the anonymous referee for providing many
thoughtful suggestions to improve the paper.

\end{document}